\documentclass[11pt]{amsart}
\usepackage{amsmath,amssymb,mathtools}
\usepackage[T1]{fontenc}
\usepackage{lmodern,textcomp}
\usepackage{color}
\usepackage{graphicx}
\usepackage{amsthm}
\usepackage[all]{xy}
\usepackage{caption}
\usepackage{enumitem}
\usepackage{hyperref}

\usepackage{fullpage}

\usepackage{setspace}

\def\M{{\mathcal{M}}}

\def\oM{\overline{\mathcal{M}}}

\def\Om{{\Omega}}
\def\oOm{\overline{\Omega}}

\def\RR{\mathbb{R}}

\def\NN{\mathbb{N}}
\def\QQ{\mathbb{Q}}
\def\PP{\mathbb{P}}

\def\max{{\rm max}}

\def\sskip{\vspace{4pt}}

\def\oGamma{{\overline{\Gamma}}}

\theoremstyle{definition}
\newtheorem{definition}{Definition}[section]

\theoremstyle{plain}
\newtheorem{conjecture}[definition]{Conjecture}
\newtheorem{theorem}[definition]{Theorem}
\newtheorem{claim}[definition]{Claim}

\newtheorem{proposition}[definition]{Proposition}

\newtheorem{lemma}[definition]{Lemma}

\begin{document}

\baselineskip=17pt

\title{A FLAT PERSPECTIVE ON MODULI SPACES OF HYPERBOLIC SURFACES}
\date{\today}
\author{Adrien Sauvaget}
\address{CNRS, Universit\'e de Cergy-Pontoise, Laboratoire AGM, UMR 8088, 2 av. Adolphe Chauvin 95302 Cergy-Pontoise Cedex, France}
\email{adrien.sauvaget@math.cnrs.fr}

\maketitle

\begin{abstract} Volumes of moduli spaces of hyperbolic cone surfaces were previously defined and computed when the angles of the cone singularities are at most $2\pi$.  We propose a general definition of these volumes  without restriction on the angles. 
This construction is based on flat geometry as our proposed volume is a limit of Masur--Veech volumes of moduli spaces of multi-differentials.  This idea generalizes the observation in quantum gravity  that the Jackiw--Teitelboim partition function is a limit of minimal string partition functions from Liouville gravity. Finally, we use the properties of these volumes to recover Mirzakhani's recursion formula for Weil--Petersson polynomials. This provides a new proof of Witten--Kontsevich's theorem.


 \end{abstract}

\setcounter{tocdepth}{1}
\tableofcontents

\section{Introduction}


\subsection{Hyperbolic cone surfaces} Let $(g,n)\in \NN^2$ be a pair satisfying  $2g-2+n>0$. We set 
\begin{equation}
\Delta_{g,n}\coloneqq \left\{ a=(a_1,\ldots,a_n)\in  \RR_{\geq 0}^n, \text{ such that }\,  |a|\coloneqq \sum_{i=1}^n a_i <2g-2+n \right\}.
\end{equation}
Let $a$ be a vector in $\Delta_{g,n}$. We denote by $\M_{g,n}(a)$ the moduli space of hyperbolic surfaces of genus $g$ with $n$ ordered cone singularities with angles $2\pi a_1,\ldots,2\pi a_n$ ({\em surfaces of type $a$} for short in the text), with the convention that $a_i=0$ stands for a cusp singularity. It is a smooth orbifold that carries a canonical symplectic form $\omega_{g,n}(a)$, the {\em Weil--Petersson} form. We define the Weil--Petersson volume of this space as
\begin{equation}
   V^{WP}_{g,n}(a)\coloneqq \frac{1}{(3g-3+n)!} \int_{\M_{g,n}(a)}  \left(\frac{\omega_{g,n}(a)}{4\pi^2}\right)^{3g-3+n}. 
\end{equation}
For all $a$, the space $\M_{g,n}(a)$ is real isomorphic to $\M_{g,n}$,  the moduli space of smooth curves of genus $g$ with $n$ marked points~\cite{Tro1}. However, the symplectic geometry of $\oM_{g,n}(a)$ depends on $a$, and this dependence is expected to be regular in chambers of $\Delta_{g,n}$ delimited by affine walls. We first recall previous results obtained for small angles.


\subsubsection*{Angles smaller than $\pi$} If we assume that the coordinates of $a$ are smaller than $1/2$, then the space $\M_{g,n}(a)$ shares common features with the moduli spaces of hyperbolic surfaces with geodesic boundaries described by Mirzkhani in~\cite{Mir1,Mir}. For instance, $\M_{g,n}(a)$ carries canonical systems of Darboux coordinates, the {\em Frenchel--Nielsen} coordinates. The Weil--Petersson form extends to a kähler form on $\oM_{g,n}$, the compactification of $\M_{g,n}$ by stable curves.  
The cohomology class of this form is  given by
\begin{equation}
    \frac{[\omega_{g,n}(a)]}{2\pi^2}= \kappa_1 - \sum_{i=1}^n a_1^2\psi_i,
\end{equation}
where $\psi_i$ is the Chern class of the co-tangent line at the $i$-th marked point and $\kappa_1$ is the Mumford-Morita class~\cite{Mir,TanWonZha,DoNor,AnaNor}. In particular, the Weil--Petersson volume is a rational polynomial defined by integrals of tautological classes 
\begin{equation}
    (-1)^{g-1+n} V^{WP}_{g,n}(a)=  P_{g,n}(a) \coloneqq \frac{1}{(3g-3+n)!} \int_{\oM_{g,n}} \left(-\frac{1}{2}\kappa_1 + \sum_{i=1}^n \frac{a_1^2}{2}\psi_i\right)^{3g-3+n}.
\end{equation}
The polynomials $P_{g,n}$ are often called {\em Mirzakhani polynomials}\footnote{Our convention for Mirzakhani polynomials differs from the classical one by a sign and a factor $(2\pi)^{6g-6+2n}$.}.

\subsubsection*{Angles smaller than $2\pi$} Let $a$ be a vector in $[0,1[^n$. Anagnostou--Mullane--Norbury recently proved that the Weil--Petersson symplectic form extends  to a kähler form if we replace $\oM_{g,n}$ by $\oM_{g,n}(a)$, Hasset's compactification of $\M_{g,n}$ by $a$-stable curves~\cite{AnaMulNor,Has}.  The spaces $\oM_{g,n}(a)$ are birational models of $\oM_{g,n}$ that are constant in chambers of $[0,1[^n$, and $\oM_{g,n}$ is the model associated with the chamber defined by $a_i+a_j<1$ for all pairs $(i,j)$.  As a consequence, the Weil-Petersson volume is a piecewise polynomial function in $a$, and they described the wall-crossing formulas explicitly. Finally, they observed that this volume tends to 0 as $a$ goes to $1$. This is the last of a long series of works that computed Weil--Petersson volumes for $a$ in several sub-domains of $]0,1[^n$ through different approaches~\cite{DoNor,SchTra,MaxTur}.

\sskip

\noindent {\em What about general angle data?} The primary motivation for the present work is the following problem:
{\em  How to define  the volume of the space $\M_{g,n}(a)$ for a general value of $a$ in $\Delta_{g,n}$? Can we compute this volume function?} We propose indirect but explicit solutions to this problem based on the geometry of moduli spaces of differentials. In the course of our construction, we will also propose explicit conjectural expression of the cohomology class of the Weil-Petersson symplectic form.

 \subsection{Moduli spaces of differentials} Let $(a,k)$ be a pair of $\Delta_{g,n}\times \NN^*$. We denote by $\Omega_{g,n}(a,k)$, 
 the {\em  moduli space of $k$-differentials of type $a$}, i.e. the moduli space of tuples $(C, x_1,\ldots, x_n, \eta)$, where: \begin{itemize}
     \item $(C,x_1,\ldots,x_n)$ is a smooth curve of genus $g$ with $n$ markings,
     \item $\eta$ is a meromorphic  $k$-differential on $C$ with singularities (zeros or poles) of order at least $ka_i-k$ at $x_i$  for all $i\in \{1,\ldots,n\}$,  and no poles outside the markings. 
 \end{itemize} 

 \sskip

 For $k$ large enough, the space $\Omega_{g,n}(a,k)$ is a vector bundle over $\M_{g,n}$. Let $r$ be the rank of this vector bundle and  $p\colon\PP\Om_{g,n}(a,k)\to \M_{g,n}$ its projectivization.  
 The line bundle $\mathcal{O}(-1)\to \PP\Om_{g,n}(a,k)$ carries a natural hermitian metric $h_{a,k}$, the {\em area metric}, defined as follows: a $k$-differential determines a flat metric with cone singularities on the underlying surface, and the value of $h_{a,k}$ is the $k$-th power of the area for this metric.  We denote by ${\alpha}_{g,n}(a,k)$ the curvature form associated with the dual of the area metric. 
 \begin{definition}
 Let $d\geq 0$, and let $\eta$ be a $C^\infty$- form of co-degree $2d$ in $\oM_{g,n}$. We set
\begin{equation}
     \omega_{g,n,d}(a)(\eta)\coloneqq \limsup_{k\to \infty}  \frac{1}{(-k^2)^{d}}\int_{\Omega_{g,n}(a,k)} \alpha_{g,n}(a,k)^{r-1+d}\wedge p^*\left(\eta|_{\M_{g,n}}\right).
\end{equation}
(the convergence of the integrals in the RHS was proved in~\cite{CosMoeZac}).
 \end{definition}
\begin{conjecture}\label{conj}
    For all $d$, 
    $\omega_{g,n,d}(a)$ is the current defined by integration of $\frac{1}{d!}\left(\frac{\omega_{g,n}(a)}{4\pi^2}\right)^d$.
\end{conjecture}

 For $n=0$, the restriction of this conjecture to forms with compact support in $\M_{g}$ follows from the work of Ma--Zhang~\cite{MaZha}. The first technical difficulty in extending their arguments to $n>0$ is the presence of a continuous part in the spectrum of the Laplacian of cone surfaces. However, the most delicate part of this conjecture is to extend these results to differential forms on families of singular curves.  In particular, we formulated this conjecture on $\oM_{g,n}$, although we expect that the Weil--Petersson form should extend naturally (with singularities) to an alternative compactification $\oM_{g,n}(a)$ of $\M_{g,n}$ that would generalize Hasset's moduli spaces of $a$-stable curves. The construction of this space for a general vector $a$ is an open problem.

\sskip

Here, we consider only the cohomological counterpart of Conjecture~\ref{conj}. Namely we define $s_{g,n,d}(a,k)$ to be the cohomology class in $H^*(\oM_{g,n},\QQ)$ of the current defined by  integration of $\alpha_{g,n}(a,k)^{r-1+d}$. 
\begin{claim}\label{claim}
For all $d$, the functions $k^{-2d}s_{g,n,d}(\cdot,k)\colon \Delta_{g,n}\to H^{2d}(\oM_{g,n},\QQ)$ converge uniformly towards a function $s_{g,n,d}$ as $k$ goes to $\infty$. The function $s_{g,n,d}$ is a continuous piece-wise polynomial in $a$ of degree $2d$ with coefficients in the tautological ring of $\oM_{g,n}$  that can be explicitly computed. Moreover for all $a\in ]0,1[^n$, we have 
\begin{equation}
    {\sum_{d\geq 1} s_{g,n,d}(a)}
={\rm exp}\left(-\frac{[\omega_{g,n}(a)]}{4\pi^2}\right) \text{ in $H^*(\oM_{g,n},\QQ)$.}
\end{equation}
\end{claim}
We will prove this claim in full generality in a subsequent paper. Here, we restrict our attention to smaller domains of angle data where the analysis is simplified. For all $x> 0$ we denote by $\Delta_{g,n}^{\leq x}\subset \Delta_{g,n}$
the set of vectors of $\Delta_{g,n}$ such that $a_{n}\leq x$ and $a_i<1/2$ for all $i\in \{1,\ldots,n-1\}$. 
\begin{theorem}\label{th:main1}
    The restriction of  Claim~\ref{claim} to $a\in \Delta_{g,n}^{\leq 2}$ is valid. 
\end{theorem}
The functions $s_{g,n,d}$ are computed by induction on $g$ and $n$ (see Theorem~\ref{th:main2}).
 The main interest of this theorem is to make $s_{g,n,d}(a)$ explicit and, therefore, provides a conjectural expression of powers of the Weil-Petersson symplectic form. If we restrict to the numerical counterpart, then we can consider the function on $\Delta_{g,n}^{\leq 2}$ defined by
\begin{equation}
    V_{g,n}(a)\coloneqq\int_{\oM_{g,n}} (-1)^{g-1+n}s_{g,n,3g-3+n}(a).
\end{equation}
By Theorem~\ref{th:main1}, this function is a piece-wise polynomial of degree $6g-6+n$ with rational coefficients, and we have $V_{g,n}(a)=V^{WP}_{g,n}(a)$ for all $a\in \Delta_{g,n}^{\leq 1}$. The intuition behind the construction of this function comes from the study of flat surfaces.  The function $V_{g,n}$ is an integral on moduli space of $k$-differentials with $n$ cone singularities of angles prescribed by $a$ and simple zeros, i.e., cone singularities of angles $\left(1+ \frac{1}{k}\right) 2\pi$~\cite{SauFlat}.  Heuristically,  the simple zeros form a small punctual negative curvature,  and as $k$ goes to infinity, these small singularities equidistribute to approximate a smooth metric of constant negative curvature. 

\sskip

This heuristic convergence has already been observed and used in theoretical physics. The Weil--Petersson measure is used to compute the partition functions in Jackiw--Teitelboim (JT) theory, a gravity theory with dilaton~\cite{WitJT}. These partition functions are limits of partition functions for $(2,k)$ minimal string models in  Liouville conformal field theory that can be expressed in terms of the Segre classes of moduli spaces of $k$-differentials~\cite{LiouvilleJT}. Then, wall-crossing formulas in JT theory with angles smaller than $2\pi$ are recovered from the study of the change of regime for instantons in the $(2,k)$-minimal models~\cite{MaxTur}. Theorem~\ref{th:main1} generalizes these results from a mathematical perspective.

\subsection{Isomonodromic foliations} If $a_n\neq 1$ or $2$, then we set
\begin{equation} 
{\rm Vol}_{g,n}(a)\coloneqq \frac{V_{g,n}(a)}{\prod_{i=1}^n \sin(a_i\pi)}.
\end{equation}
\begin{theorem}\label{th:volgn} If $2g-2+n\geq 2$, then
    the function ${\rm Vol}_{g,n}$ extends to a continuous non-negative function on $\Delta_{g,n}^{\leq 2}$. In particular, $V_{g,n}(a)=0$ if $a_n=1$ or $2$.\footnote{The fact that $V_{g,n}$ may be negative is due to the difference between the orientation defined by the complex structure (used implicitly for integration) and the one defined by the symplectic structures (Weil--Petersson or area form).}
\end{theorem}

\begin{figure}
    \centering
    (a)\includegraphics[scale=0.4]{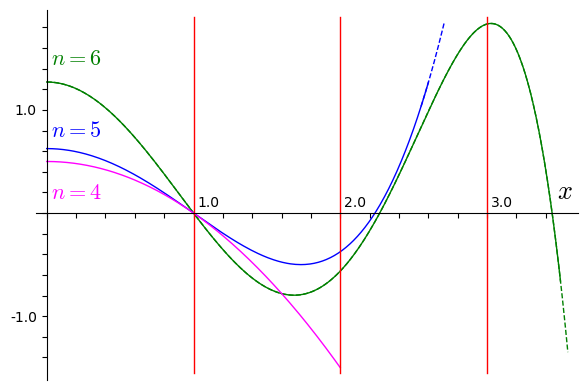}
     (b)\includegraphics[scale=0.4]{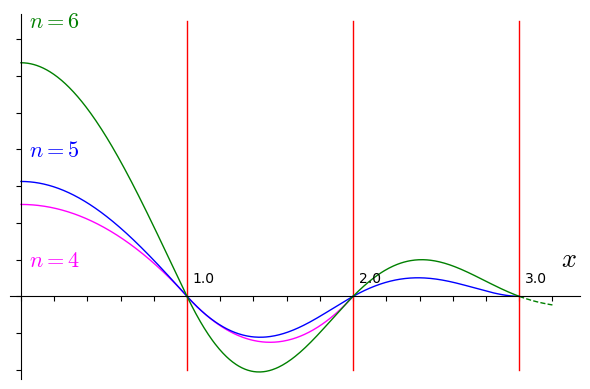} 
     
     (c)\includegraphics[scale=0.4]{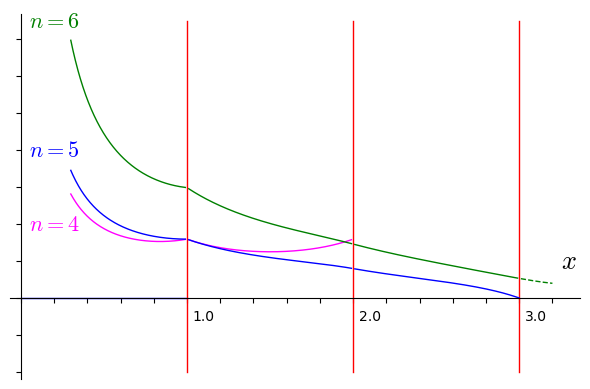}
    \caption{Graphs of the functions $x\mapsto (-1)^{n-1}P_{0,n}(x)$ (a), $V_{0,n}(x)$ (b), and $\frac{1}{{\rm sin}(x)}V_{0,n}(x)$ (c) for $n=3,4,$ and $5$ (here $x$ implicitly stands for the vector $(0,\ldots,0,x)$). The values for $x>2$ are based on our forthcoming work.}
    \label{fig:graph}
\end{figure}

 The vanishing of $V_{g,n}(a)$ when $a$ has an integral coordinate is the consequence of the vanishing of the top power of $\alpha_{g,n}(a,k)$  which is proved in~\cite{SauFlat} via flat geometric arguments. If we assume that Conjecture~\ref{conj} is valid, then this vanishing can be proved directly from a hyperbolic point of view via the existence of {\em isomonodromic deformations}. Indeed, the Teichmüller space associated with $\M_{g,n}(a)$ carries a {\em monodromy} morphism to the ${\rm SL}(2,\RR)$ character variety of a punctured surface with ``relative conditions'' on the monodromy around the punctures prescribed by $a$. For general values of $a$, this morphism is locally an isomorphism, and the Weil--Petersson symplectic form should be a pull-back of a symplectic form on the character variety as was shown by Goldman when $n=0$~\cite{Gol}. However, if a coordinate of $a$ is integral, then the relative condition is trivial, and the monodromy morphism is a submersion with complex fibers of positive dimension, which are the leaves of the {\em isomonodromy foliation}.  The existence of this foliation implies that the top power of the pull-back of a symplectic form along the monodromy morphism vanishes. 
 
\sskip

As the notation suggests, the value ${\rm Vol}_{g,n}(a)$ should be considered as an alternative definition of volumes of $\oM_{g,n}(a)$. Although  ${\rm Vol}_{g,n}(a)$ can be determined from $V_{g,n}(a)$ (or conjecturally $V_{g,n}^{WP}(a)$) at generic points of $\Delta_{g,n}$, it has the advantage of being non-trivial when $a$ is integral.  To define it geometrically, one would need to construct a volume form that is non-trivial in the direction of the isomonodromic foliations. From a flat geometric perspective, this function is the analog of Masur--Veech volumes of moduli spaces of differentials or flat surfaces~\cite{Vee,Mas,CheMoeSauZag,CheMoeSau,SauFlat}.  

\subsection{Recursion formulas for Mirzakhani polynomials} The geometry of moduli spaces of hyperbolic cone surfaces serves as a guide to introduce the objects studied in the present paper.  Beyond this original motivation, the classes $s_{g,n,d}(a)$ can be used to study tautological classes of moduli spaces of stable curves. In the next paper,  we will prove that certain vanishing properties of $s_{g,n,d}(a)$ can be used to produce tautological relations. 

\sskip

Here, we focus on the properties of the functions $V_{g,n}$. We will prove that in each chamber of polynomiality, $V_{g,n}$ can be explicitly expressed in terms of Mirzkhani polynomials. Then, the vanishing of $V_{g,n}$ at $a_n=1$ or $2$ produces relations between Mirzkhani polynomials. The first family of relations ($a_n=1$) recovers the Do--Norbury relations~\cite{DoNor}.  The second family $(a_n=2)$ implies the following theorem.

\begin{theorem}\label{th:kdv} For all $(g,n)\neq (0,3)$ with $n\geq 1$, we have
\begin{eqnarray}\label{eq:kdv} \nonumber
&&\!\!\!\!\!\!\!\!\!\!\!(1+a_1)P_{g,n}(1+a_1,a_2,\ldots) - (1-a_1)P_{g,n}(1-a_1,a_2,\ldots) \\  & &= \sum_{1<i\leq n}  \int_{t=a_i-a_1}^{a_i+a_1} tP_{g,n-1}(t,a_2, \ldots, \widehat{a_i},\ldots)\, dt \\ \nonumber
&&\,\,\,\,\,\,\,\,\, + \int_{t=-a_1}^{a_1} \int_{y=0}^t \frac{t (y-t)}{2} P_{g-1,n+1}(y,t-y,a_2, \ldots) \, dy\, dt\\ \nonumber
&&\,\,\,\,\,\,\,\,\, + \sum_{\begin{smallmatrix}g_1+g_2=g\\ I_1\sqcup I_2 =\{2,\ldots,n\}\end{smallmatrix}} \int_{t=-a_1}^{a_1} \int_{y=0}^t \frac{t (y-t)}{2} P_{g_1,|I_1|+1}(y, \{a_i\}_{i\in I_1}) P_{g_2,|I_2|+1}(t-y, \{a_i\}_{i\in I_2}) \, dy\, dt.
\end{eqnarray}
(where the notation $\widehat{a_i}$ means that we remove the variable).
\end{theorem}
 By applying the Laplace transform to formula~\eqref{eq:kdv}, Du showed that it is equivalent to Mirzakhani's original recursion for the polynomials $P_{g,n}$~\cite{Du,Mir1}. It is worth noting that his approach relies on another heuristic interpretation of Weil--Petersson volumes for an angle $4\pi$.  However, his proof of Theorem~\ref{eq:kdv} reduces to Mirzkhani's original theorem while we provide an independent proof.  In particular, following the arguments of Mirzakhani, we obtain a new proof of Witten--Kontsevich's theorem: integrals of $\psi$-classes satisfy the Virasoro constraints~\cite{Wit,Kon}. If we assume that Conjecture~\ref{conj} holds, then our method shows that the Virasoro constraints (of the point) are a numerical consequence of the existence of isomonodromic foliations.

\subsection*{Acknowledgement}  I would like to thank Bertrand Deroin and Siarhei Finski for extensive discussions on the circle of ideas motivating this article.

\section{Induction formulas for $s_{g,n,d}$} \label{sec:ind}

In this section, we introduce several combinatorial structures that will be used in the rest of the text, and we state the formulas defining the functions $s_{g,n,d}$ by induction on $g$ and $n$. These formulas will be proved in Section~\ref{sec:growth}.

\subsection{Rational hyperbolic graphs} Let $\Gamma$ be a stable graph of genus $g$ with $n$ marked points (as defined in~\cite{GraPan}  for instance). 

\begin{definition} 
A {\em twist} on $\Gamma$ is a function $b:H(\Gamma)\to \RR$ satisfying the following constraints:
\begin{enumerate}
    \item For all edges $e=(h,h')$ we have $b(h)+b(h')=0$.
    \item If $(h_1,h_1')$ and $(h_2,h_2')$ are vertices between two vertices $v$ and $v'$ then $b(h_1)\geq 0\Rightarrow b(h_2)\geq 0$. In which case, we denote $v\geq v'$. \item The relation $\geq$ is transitive.
    \item If $v$ is a vertex then we denote by $b(v)$ the vector $(b(h))_{h\mapsto v}$ of twists at half-edges incident to $v$, and we have the inequality $|b(v)| \leq 2g(v)-2+n(v).$
\end{enumerate}
We define the twist at an edge as $b(e)=\sqrt{-b(h)b(h')}$ (if $e$ is the edge $(h,h')$) and the {\em multiplicity} of the twisted graph $(\Gamma,b)$ as $$m(\Gamma,b)\coloneqq \prod_{e\in E(\Gamma)} b(e).$$
We say that a twist is {\em compatible} with a vector $a\in \RR^n$ if $b(i)=a_i$ for all legs. 
\end{definition}

\begin{definition}
A {\em bi-colored graph}  is the data of 
$$\oGamma=(\Gamma,b,V^{\rm ab}\subset V, V=V_0\sqcup V_{-1})$$
where $b$ is a twist, $V^{\rm ab}$ is a set of vertices $v$ such that $b(v)$ is integral, and the partition of $V$ into two {\em levels} $V_0\sqcup V_{-1}$ satisfies: all edges of $\Gamma$ are between a vertex $v$ of $V_0$ and a vertex $v'$  $V_{-1}$, and we have $v>v'$.  
\end{definition}

\begin{definition}
A bi-colored is a {\em hyperbolic graph} if $|b(v)|=2g(v)-2+n(v)$ for all vertices $v$ in $V_{-1}$ and $V^{\rm ab}=\emptyset$. It is a {\em rational hyperbolic graph}  if $V_{-1}$ is reduced to a single vertex of genus 0.  
\end{definition}

\subsection{Expression of $s_{g,n}$ via rational graphs}
 We denote by ${\rm Rat}_{g,n}$ the set of {\em rational graphs}, i.e. graphs such that the vertex carrying the $n$-th leg is of genus 0  (the {\em central vertex}), and all edges connect this vertex to another vertex ({\em outer vertices}). 
 A hyperbolic graph structure on such graph $\Gamma$ is uniquely determined by the twist function. We denote by $\Delta_{\Gamma}(a)\subset \RR^{H(\Gamma)}$ the set of rational hyperbolic structures on $\Gamma$ compatible with $a$, i.e. the simplex  of functions $b\colon H(\Gamma) \to \RR$ satisfying the constraints:
 \begin{equation}
     \left\{
     \begin{array}{cl}
          b(h)+b(h')=0 \text{ and } b(h)\neq 0 & \text{if $(h,h')$ is an edge,}  \\
          b(h)\geq 0 & \text{if $h$ is incident to $v\in V^0$,}\\
          \sum_{h\mapsto v} b(h) < 2g(v)-2+n(v) & \text{if $v$ is a vertex in $V^0$,}\\
          \sum_{h\mapsto v} b(h) = 2g(v)-2+n(v) & \text{if $v$ is the central vertex,}\\
          b(i)=a_i & \text{for all $i\in \{1,\ldots,n\}$}.
     \end{array}\right.
 \end{equation}
 Equivalently, we can define $\Delta_{\Gamma}(a)$ as a simplex of $\RR_{>0}^{E(\Gamma)}$ if we specify only the twist at edges. This domain is empty or of dimension $h^{1}(\Gamma)$. 
 \begin{theorem}\label{th:main2} Let $s_{g,n,d}\colon \Delta_{g,n}^{\leq 2}\to H^{2d}(\oM_{g,n},\QQ)$ be the functions determined  by the following identities \begin{eqnarray}
 && \text{\it base of the induction: } s_{g,n,0}=1; \label{formain:base} \\
 &&\text{\it small angles value: } \label{formain:small}
     s_{g,n}(a)=e_{g,n}(a)\coloneqq {\rm exp}\left(-\frac{1}{2}\kappa_1 + \sum_{i=1}^n \frac{a_i^2}{2}\psi_i\right) \text{ if $a_n\leq 1/2$};\\
&&\text{\it derivative: } \label{formain:der}
     \frac{\partial}{\partial a_n} s_{g,n}=a_n\psi_n s_{g,n} - \sum_{\Gamma \in {\rm Rat}_{g,n}}  \int_{b\in \Delta_{\Gamma}(a)} \frac{m(\Gamma,b)}{|{\rm Aut}(\Gamma)|} \zeta_{\Gamma *}\left(\bigotimes_{v\in V^{\rm out}} s_{g(v),n(v)}(b(v))  \right) \, db.  
 \end{eqnarray}
  where $s_{g,n}=\sum_{d\geq 0} s_{g,n,d}$, $\zeta_{\Gamma}\colon \prod_{v \in V(\Gamma)} \oM_{g(v),n(v)}\to \oM_{g,n}$ is the gluing morphism, and $V^{\rm out}$ is the set of outer vertices of a rational graph.  The function $s_{g,n,d}$ is the absolute limit of the functions $k^{-2d}s_{g,n,d}(\cdot,k)\colon \Delta_{g,n}^{\leq 2}\to H^{2d}(\oM_{g,n},\QQ)$ as $k$ goes to infinity. 
 \end{theorem}
These formulas will be proved in Section~\ref{sec:growth}. In this section, we use this theorem to describe explicitly the function $s_{g,n}$ in terms of the polynomials $e_{g,n}$.

\subsection{Expression of $s_{g,n}$ in terms of tautological classes} We use the fact that $$\frac{\partial}{\partial a_n} \psi_n^\ell s_{g,n}=a_n\psi_n^{\ell+1}s_{g,n}+\text{ boundary terms}$$  
to compute the functions $\psi_n^\ell s_{g,n}$ by an induction on $\ell$ that goes down from $\ell=3g-3+n$. 

\sskip

Let $\Gamma$ be a rational graph in ${\rm Rat}_{g,n}$. If  $\Delta_\Gamma(a)$ is non-empty for some $a\in \Delta_{g,n}^{\leq 2}$, then the central vertex carries at most $4$ half-edges. If $\ell \geq 2$, then $\zeta_\Gamma^*\psi^{\ell}=0$ because $n$ is supported on  a vertex $v$ with $(g(v),n(v))=(0,3)$ or $(0,4)$. This implies that $\frac{\partial}{\partial a_n} (\psi_n^\ell s_{g,n})=a_n\psi^{\ell +1}_n s_{g,n},$ and 
\begin{equation}
\psi_n^\ell s_{g,n}= \psi_n^\ell e_{g,n}.
\end{equation}

The next step is to compute $\frac{\partial}{\partial a_n}\psi_1s_{g,n}$. By a similar argument, the rational graphs contributing to the expression of this derivative are the ones with 4 half-edges on the central vertex. There are 3 types of such graphs:
\begin{center}
\begin{tabular}{|c|c|c|}
\hline
      $\Gamma_{\{i,j,n\}}$  & $\Gamma^{\rm root}_{\{i,n\}}$  &  $\Gamma_{\{i,n\}}^{g_1,g_2,I_1,I_2}$  \\ 
      for $1\leq i<j<n$ & for $1\leq i<n$ & for $1\leq i<n$, $g=g_1+g_2$,\\&& and $\{1,\ldots,n\}\setminus\{i,n\}= I_1\sqcup I_2$\\
      \hline
     $
\xymatrix@C=0.4em@R=1.2em{ 
&& &&\\
&&*+[Fo]{g}\ar@{-}[d] \ar@{-}[lu]_{\ldots}\ar@{-}[llu]\ar@{-}[rru]  \\ 
&&*+[Fo]{0}, \ar@{-}[d]\ar@{-}[rd]\ar@{-}[ld]\\
&i& j & n 
}$ & $
\xymatrix@C=0.4em@R=1.2em{ 
&&& \!\!\!\!\!\!\!\!\!\!\!\!\!\!\!\!\!\!
\!\!\!\!\!\!\!\!\!\!\!\!\!\!\!\!\!\!\!\!\!\!\!\!\!\!\!\!\!\!\!\!\!\!\!\!&&\\
&&&*+[Fo]{{\!{\scriptstyle g-1}\!}}\ar@{-}@/^/[d]\ar@{-}@/_/[d] \ar@{-}[lu]_{\ldots}\ar@{-}[llu]\ar@{-}[rru]  \\ 
&&&*+[Fo]{0}, \ar@{-}[rd]\ar@{-}[ld]\\
&&i&  & n}$ & $
\xymatrix@C=0.4em@R=1.2em{ 
&&I_1&& \!\!\!\!\!\!\!\!\!\!\!\!\!\!\!\!\!\!
\!\!\!\!\!\!\!\!\!\!\!\!\!\!\!\!\!\!\!\!\!\!\!\!\!\!\!\!\!\!\!\!\!\!\!\!&&I_2&\\
&&&*+[Fo]{g_1} \ar@{-}[rrd] \ar@{-}[lu]_{\ldots}\ar@{-}[llu] &&&& *+[Fo]{g_2}\ar@{-}[lld]  \ar@{-}[lu]_{\ldots}\ar@{-}[llu] \\ 
&&&&&*+[Fo]{0}, \ar@{-}[rd]\ar@{-}[ld]\\
&&&&i&  & n
}$  
\\
\hline
\end{tabular}
\end{center}

If $a$ is in $\Delta_{g,n}^{\leq 2}(a)$, then the twists of edges of these graphs are at most $1/2$. For $\ell\geq 0$, we set
\begin{eqnarray*}
     D_{g,n}^{4,\ell}(a)\!\! &\coloneqq & \!\! \sum_{1\leq i<j<n} (a_i+a_j+a_n-2)^+ \zeta_{\Gamma_{\{i,j,n\}} *} \left(\psi_n^\ell\otimes e_{g,n-2}(\ldots,\widehat{a_i},\ldots,\widehat{a_j},\ldots,a_n+a_i+a_j-2) \right) \\
    && \!\! + \sum_{1\leq i<n}  \int_{y=0}^{(a_i+a_n-2)^+} \frac{y (a_i+a_n-2-y)}{2} \\
    && \,\,\,\,\,\,\,\,\,\, \,\,\,\,\, \zeta_{\Gamma_{\{i,n\}}^{\rm loop}*}\left(\psi_n^\ell\otimes e_{g-1,n}(\ldots,\widehat{a_i},\ldots,y,a_n+a_i-2-y)\right) dy\\
    && \!\! + \sum_{1\leq i<n} \underset{I_1\sqcup I_2=\{1,\ldots,n-1\}\setminus\{i\}}{\sum_{g_1+g_2=g}}  \int_{y=0}^{(a_i+a_n-2)^+} \frac{y (a_i+a_n-2-y)}{2}  \\
    &&\,\,\,\,\,\,\,\,\,\, \,\,\,\,
    \zeta_{\Gamma_{\{i,n\}}^{I_1,I_2,g_1,g_2}*}\left(\psi_n^\ell \otimes e_{g_1,|I_1|+1}(y,\{a_i\}_{i\in I_1})\otimes e_{g_2,|I_2|+1}(a_i+a_n-2-y,\{a_i\}_{i\in I_2})\right) dy,\\
\end{eqnarray*}
where $(x)^+={\rm max}(x,0)$. Then we have 
\begin{equation}
    \frac{\partial}{\partial a_n} (\psi_n s_{g,n})=a_n\psi^{2}_n e_{g,n}-D_{g,n}^{4,1}(a).
\end{equation}

Finally, the expression of $\frac{\partial}{\partial a_n}s_{g,n}$ also involves graphs with $3$ legs on the central vertex:
\begin{center} 
\begin{tabular}{|c|c|c|}
\hline 
      $\Gamma_{\{i,n\}}$  & $\Gamma^{\rm root}_{\{n\}}$  &  $\Gamma_{\{n\}}^{g_1,g_2,I_1,I_2}$  \\ 
      for $1\leq i<n$ &  & for $g=g_1+g_2$, and $\{1,\ldots,n-1\}= I_1\sqcup I_2$\\ \hline
     $
\xymatrix@C=0.4em@R=1.2em{ 
&& &&\\
&& *+[Fo]{g}\ar@{-}[d] \ar@{-}[lu]_{\ldots}\ar@{-}[llu]\ar@{-}[rru]  \\ 
&&*+[Fo]{0}, \ar@{-}[rd]\ar@{-}[ld]\\
&i&  & n 
}$ & $
\xymatrix@C=0.4em@R=1.2em{ 
&&& \!\!\!\!\!\!\!\!\!\!\!\!\!\!\!\!\!\!
\!\!\!\!\!\!\!\!\!\!\!\!\!\!\!\!\!\!\!\!\!\!\!\!\!\!\!\!\!\!\!\!\!\!\!\!&&\\
&&&*+[Fo]{{\!{\scriptstyle g-1}\!}}\ar@{-}@/^/[d]\ar@{-}@/_/[d] \ar@{-}[lu]_{\ldots}\ar@{-}[llu]\ar@{-}[rru]  \\ 
&&&*+[Fo]{0}, \ar@{-}[rd]\\
&&&  & n}$ & $
\xymatrix@C=0.4em@R=1.2em{ 
&&I_1&& \!\!\!\!\!\!\!\!\!\!\!\!\!\!\!\!\!\!
\!\!\!\!\!\!\!\!\!\!\!\!\!\!\!\!\!\!\!\!\!\!\!\!\!\!\!\!\!\!\!\!\!\!\!\!&&I_2&\\
&&&*+[Fo]{g_1} \ar@{-}[rrd] \ar@{-}[lu]_{\ldots}\ar@{-}[llu] &&&& *+[Fo]{g_2}\ar@{-}[lld]  \ar@{-}[lu]_{\ldots}\ar@{-}[llu] \\ 
&&&&&*+[Fo]{0}, \ar@{-}[rd]\\
&&&&&  & n
}$ \\ 
\hline 
\end{tabular}
\end{center}
 We set
\begin{eqnarray*}
     D_{g,n}^{3}(a)\!\! &\coloneqq& \!\! \sum_{1\leq i<n} (a_i+a_n-1)^+ \zeta_{\Gamma_{\{i,n\}} *} \left(1\otimes s_{g,n-2}(\ldots,\widehat{a_i},\ldots,a_n+a_i-1) \right) \\
    && \!\! +  \int_{y=0}^{(a_n-1)^+} \frac{y (a_n-1-y)}{2} \zeta_{\Gamma_{\{n\}}^{\rm loop}*}\left(1\otimes s_{g-1,n}(\ldots,y,a_n-1-y)\right) dy\\
    && \!\! + \underset{I_1\sqcup I_2=\{1,\ldots,n-1\}\setminus\{i\}}{\sum_{g_1+g_2=g}}  \int_{y=0}^{(a_n-1)^+} \frac{y (a_n-1-y)}{2}  \\
    &&\,\,\,\,\,\,\,\,\,\, \,\,\,\,
    \zeta_{\Gamma_{\{n\}}^{I_1,I_2,g_1,g_2}*}\left(1 \otimes s_{g_1,|I_1|+1}(y,\{a_i\}_{i\in I_1})\otimes s_{g_2,|I_2|+1}(a_n-1-y,\{a_i\}_{i\in I_2})\right) dy.\\
\end{eqnarray*}
Altogether, we have the following expression:
\begin{equation}
    \frac{\partial}{\partial a_n}  s_{g,n}=a_n\psi_n s_{g,n}-D_{g,n}^3(a)-D_{g,n}^{4,0}(a).
\end{equation}
The the function $s_{g,n,d}$ is a piece-wise polynomial of degree $2d$ defined by 
\begin{equation}
     s_{g,n}(a)=e_{g,n}(a[0])+\int_{t=0}^{a_n} \frac{\partial}{\partial a_n} s_{g,n}(a[t]) dt,
 \end{equation}
 where $a[t]$ is the vector obtained from $a$ by replacing $a_n$ by $t\in \RR_{>0}$. Moreover, $s_{g,n,d}$ is of class $C^1$ as $D_{g,n}^{3}$ and $D_{g,n}^{4,\ell}$ are continuous.

\section{Tautological calculus in moduli spaces of differentials}\label{sec:taut}

Here, we recall some elements of intersection theory on moduli spaces of multi-differentials from~\cite{Sau,BCGGM2,SauFlat}. The main result that will be used in the next sections is Proposition~\ref{pr:main}. 

\subsection{Incidence variety compactification} A {\em rational pair}  $(a,k)$ of a subset $E$ of 
$\RR^n$ is the data  of a rational vector $a$ in $E$ and a positive integer $k$ such that $ka$ is integral. We denote by $\QQ\PP(E)$ the set of rational pairs in $E$. If $(a,k)$ is a rational pair, then a {\em $k$-twisted graph} $(\Gamma,b)$ (compatible with $a$) is a twisted graph such that $kb$ takes integral values. 

\sskip

Let $(a,k)$ be a rational pair of $\RR^n$, and let $P=(p_1,\ldots,p_n)$ be a vector of positive integers such that $p_i\geq k|a_i|$ for all $i\in \{1,\ldots,n\}$. We denote by $\oOm_{g,n}^{k,P}\to \oM_{g,n}$ the vector bundle with fiber
$$
H^0\left(C,\omega_{\rm log}^{\otimes k}(p_1x_1+\ldots+
p_nx_n\right)$$
(where $\omega_{\rm log}$ is the log-dualizing sheaf $\omega_C(x_1+\ldots+x_n)$). The space $\Om_{g,n}(a,k)$ of $k$-differentials of type $a$  is a sub-cone of $\oOm_{g,n}^{k,P}$. We denote by $\oOm_{g,n}(a,k)$ the closure of $\Om_{g,n}(a,k)$ in $\oOm_{g,n}^{k,P}$. This compactification does not depend on the choice of $P$ and is called the {\em incidence variety compactification}. 

\sskip

We assume here that $a$ is non-negative and  $|a|=2g-2+n$. 
Let $\oGamma$ be a $k$-bi-colored graph compatible with $a$. This graph determines a boundary component $\oOm_{g,n}(a,k) 
$ of co-dimension at least $1$. To construct it, we consider the cones
$$
\widetilde{\Om}_{\oGamma}(k)_{i} \coloneqq \prod_{v \in V_{i}\setminus (V_{i}\cap V^{\rm ab})} \oOm^{\rm nab}_{g,n}(a,k) \prod_{v \in V_{i}\cap V^{\rm ab}} \oOm_{g,n}^{\rm ab}(a,k)
$$
where $\oOm_{g,n}^{\rm ab}(a,k)$ and $\oOm_{g,n}^{\rm nab}(a,k)$ are the closure of the loci of $k$-differentials on smooth curves obtained (respectively not obtained) as $k$-th power of a meromorphic one-form, and 
$$
p_i\colon \widetilde{\Om}_{\oGamma}(k)_{i}\to \oM_{i}\coloneqq \prod_{v \in V_i} \oM_{g(v),n(v)}
$$
is defined by forgetting the differential. At level 0, we set $\oOm_{\oGamma}(k)_{0}= \widetilde{\Om}_{\oGamma}(k)_{0}$. At level $-1$, we denote by $\oOm_{\oGamma}(k)_{-1}$ the sub-locus of $\widetilde{\Om}_{\oGamma}(k)_{-1}$ defined by the {\em global residue condition} (GRC) of~\cite{BCGGM2}. We do not state the GRC here, but we will use the two following facts from \cite{BCGGM2}:
\begin{enumerate}
\item The co-dimension of $\oOm_{\oGamma}(k)_{-1}$ in $\widetilde{\Om}_{\oGamma}(k)_{-1}$ is at most the number of vertices in $V_0\cap V^{\rm ab}$.
\item If $V_0\cap V^{\rm ab}$ is empty then the map $p_{-1}$ has fiber of positive dimension along $\oOm_{\oGamma}(k)_{-1}$ unless $V_{-1}$ is of size $1$.
\end{enumerate}
We set $$
\oOm_{\oGamma}(k)\coloneqq \oOm_{\oGamma}(k)_{0}\times \PP\oOm_{\oGamma}(k)_{-1}.$$ There is a canonical morphism
$
\zeta_{\oGamma,k}\colon \PP\oOm_\oGamma(k)\to \PP\oOm(a,k)
$
defined by gluing the curves along nodes and imposing that the differential vanishes on vertices in $V_{-1}.$ With this notation, the space $\oOm_{\oGamma}(k)$ is a cone over $\oM_{\Gamma}$ and the following diagram commutes
$$
\xymatrix@C=4pc{
\oOm_{\oGamma}(k) \ar[r]^{\zeta_{\oGamma,k}} \ar[d]_{(p_0\times {\rm Id}_{\oM_{-1})}} & \oOm_{g,n}(a,k) \ar[d]^p \\
\oM_\Gamma \ar[r]_{\zeta_{\Gamma}}& \oM_{g,n}.
}
$$

\subsection{Adding simple zeros} Let $(a,k)$ be a rational pair in $ \Delta_{g,n}^{<2}$. In the rest of the section we assume that $k> 4(g+1)$ and $(g,n,a)\neq (1,1,(1-1/k))$. We denote  
$$\overline{a}^k=(a_1,\ldots,a_n,\underset{N(a,k)\times }{\underbrace{(1+1/k),\ldots,(1+1/k)}}),$$ 
where $N(a,k)=k(2g-2+n-|a|)$. We denote by $\pi(a,k)\colon \oOm_{g,n+N(a,k)}(\overline{a}^k,k)\to \oOm_{g,n}(a,k),$
the  morphism defined by forgetting the $N(a,k)$ marked points. It is dominant and of degree $N(a,k)!$.

\begin{lemma}\label{lem:forget} Let $\oGamma$ be a bi-colored graph compatible with $\overline{a}^k$ such that $n$ is incident to a vertex in $V_{-1}$. Let $X$ be an irreducible component of $\PP\oOm_\oGamma(k)_{-1}$ then 
$(\pi(a,k) \circ \zeta_{\Gamma,k})_* [X]=0$ unless $V^{\rm ab}$ is empty and $V_{-1}$ is reduced to a unique vertex of genus 0 and either
\begin{enumerate}
    \item all the markings forgotten by $\pi(a,k)$ are incident to the vertex in $V_0$,
    \item or this vertex has one edge, no legs in $\{1,\ldots,n-1\}$ and one leg forgotten by $\pi(a,k)$.
\end{enumerate}
\end{lemma}

\begin{proof} 
Let $\oGamma$ be a bi-colored graph. 
For each vertex, we denote by $N(v)$ the number of legs incident to $v$ which are forgotten by $\pi(a,k)$ and by $N_{-1}= \sum_{v\in \ell^{-1}(-1)} N(v)$. We will show that  $N_{-1}$ is bigger than the co-dimension of $p_{-1}\left(\PP\oOm_\oGamma(k)_{-1}\right)$ in $\oM_{-1}$ and thus $(\pi(a,k) \circ \zeta_{\Gamma,k})_*[X]=0$ is trivial for all irreducible component of $\PP\oOm_\oGamma(k)$ unless $\oGamma$ satisfies the conditions of the lemma.   

\sskip

\noindent{\em {\smash{Step 1:}} no vertex in $V_{-1}$ is of positive genus.} The vertices of $V_0\cap V^{\rm ab}$  are of genus at least one. This set is of size at most $g$. Besides, we have
$$
{\rm dim}\PP\widetilde{\Om}_\oGamma(k)_{-1} = {\rm dim}\oM_{-1} +1-\sum_{v\in V_{-1}} (g(v)-1).$$  Therefore the co-dimension of $p_{-1}\left(\PP\oOm_\oGamma(k)_{-1}\right)$ in $\oM_{-1}$ is at most
\begin{eqnarray}
   g-1+\sum_{v\in V_{-1}} (g(v)-1) < 2g.
\end{eqnarray}
If $v$ is a vertex of level $-1$ then 
\begin{eqnarray}\label{rel:Nv}
    2g(v)-2 = N(v)/k + \underset{i\mapsto v}{\sum_{i\in \{1,n\}}} (a_i-1) + \underset{h\mapsto v}{\sum_{(h,h')\in E(\Gamma)}} (b(h)-1).
\end{eqnarray}
The first sum is over the legs in $\{1,\ldots,n\}$ incident to $v$, and this sum is smaller than $1$ if $n$ is the only element in this set and $1/2$ otherwise. The second sum is over half-edges incident to $v$ and is smaller than $-1$ (as the sum is non-empty and $b(h)$ is negative). Therefore we have
\begin{equation}
N(v)/k> 2g(v)-1- \underset{i\mapsto v}{\sum_{i\in \{1,n\}}} (a_i-1) 
\end{equation}
and $N(v)> k/2$ unless $v$ is of genus 0, or $v$ is of genus $1$, with only two other half-edges ( the marking $n$ and a half of an edge to the upper vertex).  If $\oGamma$ has a vertex that does not satisfy one of these conditions, then $N_{-1}$ is bigger than co-dimension  $p_{-1}\left(\PP\oOm_\oGamma(k)_{-1}\right)$ in $\oM_{-1}$.

\sskip

To complete the first step, we need to exclude the possibility of a vertex $v$ of genus 1 in $V_{-1}$. To do so, we remark that if the markings $n$ and half of an edge are incident to this vertex, then $N(v)$ is at least 2. As there is only one edge to the vertex of level $0$, we have two possibilities:
\begin{itemize}
    \item If the GRC is trivial then $v$ is the only vertex of level $-1$ and the co-dimension of the level $-1$ in the moduli space of curves is $g(v)=1$ which is smaller than $N(v)$.
    \item If the GRC is non-trivial then $b(h)\leq -1$ for an   half-edge $h$ to an upper vertex in $V^{\rm ab}$) and thus $N(v)>k/2$ as above. 
\end{itemize}

\noindent{\em {\smash{Step 2:}} $V^{\rm ab}$ is empty.} 
Each vertex in $V_{-1}$ contains at least one leg, and the vector $\overline{a}^k$ has no integral coordinate apart from $a_n$ that can be equal to $1$. A vertex carrying the $n$-th leg necessarily has another leg if this is true. Therefore, all vertices of $V^{\rm ab}$ are of level 0. 
\sskip

Let $v_0$ be a vertex in  $V^{\rm ab}$. If this vertex is of genus at least $2$, then $b(h)\geq 2$ for at least one half-edge incident to $v_0$ and going to a vertex $v$ of level 0. For this vertex $v$ we again find that $N(v)>k/2$ because the second sum in~\eqref{rel:Nv} is at most $-3$. Besides, if we assume that there are at least $2$ vertices of level $0$, then at least one vertex $v$ of level $-1$ is connected to $v_0$ and another vertex. We again find that the second sum in~\eqref{rel:Nv} is at most $-3$. Therefore, $v_0$ is the only vertex of level $0$ of genus $1$. If this vertex is connected twice to a vertex $v$ of level $-1$, then $N(v)>k/2$ by a similar argument. Therefore, we must have $g=1$. 

\sskip

To exclude the possibility that $g=1$, we remark that each vertex of level $-1$ must carry at least one leg in $\{1,\ldots,n\}$. If $n>1$ and $v$ (of level $-1$) carrying the first leg, then the first sum in~\eqref{rel:Nv} is smaller than $1/2$ while the second one is $-2$ so $N(v)>k/2$ again. Therefore, $(g,n)=(1,1)$ and the graph $\oGamma$ must have exactly one edge connecting a vertex of genus $1$ to a vertex of genus $0$ with all legs. This situation cannot occur because the co-dimension of the level $-1$ space is $1$, and we have excluded the case $N(a,k)=1$ in our assumption for the section.

\sskip

\noindent{\em \smash{Step 3:} end of the proof.} The GRC is trivial, so there can be only a vertex of level $-1$; otherwise, $\zeta_{\oGamma,k}$ has fibers of positive dimension. Finally, the co-dimension of  $p_{-1}\left(\PP\oOm_{\oGamma}(k)_{-1}\right)$ in the moduli space of curves is 0, so the unique vertex of level $0$ must satisfy $N(v)=0$, or it has to be contracted by the forgetful morphism. In the latter situation, it carries exactly the leg $n$, one half-edge, and a forgotten leg.
\end{proof}

\subsection{Tautological relations} Let $\xi$ be the first Chern Class of $\mathcal{O}(1)$ in $\PP\oOm(a,k)$. By~\cite{CosMoeZac}, for all $m\geq 0$ we $p_*\xi^m \in H^*(\oM_{g,n},\QQ)$ is the cohomology class of the current $p_*\alpha_{g,n}(a)^d$ defined in the introduction. In order, to compute this cohomology class, for $a\in \Delta_{g,n}^{\leq 2}$ we will apply~\cite[Theorem 2.7]{SauFlat}. However, this theorem is stated for $|a|=2g-2+n$. To use it in our setting, we use the forgetful morphism  $\pi(a,k)$. 
\sskip

To state our main proposition of the section, we first define boundary components of $\PP\oOm_{g,n}(a,k)$ associated with rational graphs. If $\Gamma$ is a rational graph, then we denote by $\Delta_{\Gamma}(a,k)\subset \Delta_{\Gamma}(a)$ the set $k$-rational hyperbolic structures on $\Gamma$.  If $(\Gamma,b)$ is a rational hyperbolic graph, then we denote
$$
\oOm_{\Gamma}(k)\coloneqq \oM_{0,n(v_c)} \times \prod_{v\in V^{\rm Out}} \oOm^{\rm nab}_{g(v),n(v)}(b(v),k) 
$$
where $v_c$ is the central vertex of the graph and the morphism $\zeta_{(\Gamma,b),k}\colon \oOm_{\Gamma}(k) \to \oOm_{g,n}(a,k) 
$ 
is defined by gluing the curves along nodes and imposing that the differential vanishes on the central vertex. The image of this morphism is contained in the boundary $\oOm_{g,n}(a,k)$ because the GRC is trivial for this space of differentials. Besides, the morphism $\zeta_{(\Gamma,b),k}$ is finite of degree $|{\rm Aut}(\Gamma,b)|$.
\begin{proposition}\label{pr:main}  The following relation holds
\begin{equation}\label{for:intmain}
    \xi+(ka_n) \psi_n = [\PP\oOm_{g,n}(a[a_n+1/k],k)] + \sum_{\Gamma\in {\rm Rat}_{g,n}}\sum_{b\in \Delta_{\Gamma}(a,k)} \frac{m(\Gamma,b)k^{|E(\Gamma)|}}{|{\rm Aut}(\Gamma)|} \zeta_{(\Gamma,b),k\, *}[\PP\oOm_{(\Gamma,b)}(k)],
\end{equation}
where we recall that $a[t]$ is the vector obtained from $a$ by replacing the last coordinate by $t$, and $\PP\oOm_{g,n}(a[a_n+1/k],k)$ is considered as a sub-stack of $\PP\oOm_{g,n}(a,k)$. 
\end{proposition}

\begin{proof} We apply~\cite[Theorem 2.7]{SauFlat} for the vector $\overline{a}^k$. Then, we obtain a formula of the following type:
$$
(\xi+a_n\psi_n) = \underset{\text{compatible with $\overline{a}^k$}}{\sum_{\text{$\oGamma$, $k$-bi-colored graph}}} \text{contribution of $\oGamma$}.
$$
If we apply the push-forward along $\pi(a,k)$ to this formula, then the contribution of each graph in the sum is trivial unless it satisfies the conditions of Lemma~\ref{lem:forget}. If $R$ denotes the set of graphs satisfying these conditions, then we obtain the following formula
\begin{equation}
    \pi(a,k)_*(\xi+ka_n\psi_n) = \sum_{(\Gamma,b)\in R} \frac{m(\Gamma,b)k^{|E(\Gamma)|}}{|{\rm Aut}(\Gamma,b)|} \pi(a,k)_*\zeta_{(\Gamma,b),k\, *}[\PP\oOm_{(\Gamma,b)}(k)]. 
\end{equation}
Therefore, we need to describe the push-forward of each term of this formula to get~\eqref{for:intmain}. First, we remark that $\xi$ is the pull-back of $\xi$ along $\pi(a,k)$ so 
\begin{equation}\label{for:pushxi}
    \pi(a,k)_*\xi=N(a,k)! \xi.
\end{equation}
To compute the push-forward of $\psi_n$ we recall that 
\begin{equation}\label{for:pullpsi}
    \psi_n=\psi_n^*+ \underset{I\neq \emptyset}{\sum_{I\subset \{n+1,\ldots,n+N(a,k)\}}} \delta_{\{n\}\cup I}
\end{equation}
where $\delta_E$ is the divisor of curves with a genus 0 component carrying the markings in $E$.  The intersection of $\delta_{\{n\}\cup I}$ with $\PP\oOm_{g,n}(a,k)$ is the boundary stratum defined by the unique twist function that one can put of the stable graph defining $\delta_{\{n\}\cup I}$ (and the intersection is transverse). Then, we have
\begin{equation}\label{for:pushrational}
    \pi(a,k)_*\left(\delta_{\{n\}}\cap [\PP\oOm_{g,n}(a,k)]\right) =
\left\{\begin{array}{cl} 0 & \text{if $|I|>1$,} \\ {[}\PP\oOm_{g,n}(a[a_n+1/k],k){]} & \text{otherwise.}\end{array} \right.
\end{equation}
Combining~\eqref{for:pullpsi} and~\eqref{for:pushrational} we obtain
\begin{equation}\label{for:pushpsi}
    \pi(a,n)_*\psi_n=N(a,k)![\PP\oOm_{g,n}(a[a_n+1/k],k)].
\end{equation} 
Finally, we remark that the set $R$ splits into  $R_1\sqcup R_2$ according to the two possibilities in Lemma~\ref{lem:forget}. A stratum associated to a graph in $R_2$ is again the intersection of a class $\delta_{\{n,n+i\}}$ for some leg $i\in \{1,\ldots,N(a,k)\}$. Thus
\begin{equation}\label{for:pushR2}\sum_{(\Gamma,b)\in R_2} \frac{m(\Gamma,b)k^{|E(\Gamma)|}}{|{\rm Aut}(\Gamma,b)|} \pi(a,k)_*\zeta_{(\Gamma,b),k\, *}[\PP\oOm_{(\Gamma,b)}(k)] = N(a,k)!(ka_n+1)[\PP\oOm_{g,n}(a[a_n+1/k],k)].\end{equation}
If $(\Gamma',b')$ is a graph in $R_1$, then the image of $\PP\oOm_{(\Gamma',b')}(k)$ under  $\pi(a,k)$ is $\PP\oOm_{(\Gamma,b)}(k)$ for some $k$-rational hyperbolic graph  $(\Gamma,b)$ compatible with $a$. Moreover, the restriction of $\pi(a,k)$ to $\PP\oOm_{(\Gamma',b')}(k)$ is finite of degree $\prod_{v\in V_0} N(v)!$ (where $N(v)=k(2g(v)-2+n(v)-|b(v)|$ is the number of markings forgotten by $\pi(a,k)$). Conversely, given a rational hyperbolic graph  $(\Gamma,b)$, a graph in $R_1$ is uniquely determined by a partition of $N(a,k)$ into subsets of size $N(v)$ for each vertex. Therefore, we have
\begin{eqnarray}&&\label{for:pushR1}\!\!\!\!\!\!\!\!\!\!\!\!\!\!\!\!\!\! \sum_{(\Gamma',b')\in R_1} \frac{m(\Gamma',b')k^{|E(\Gamma)|}}{|{\rm Aut}(\Gamma',b')|} \pi(a,k)_*\zeta_{(\Gamma',b'),k\, *}[\PP\oOm_{(\Gamma',b')}(k)] \\ \nonumber &=&  \sum_{\Gamma\in {\rm Rat}_{g,n}}\sum_{b\in \Delta_{\Gamma}(a,k)} \frac{m(\Gamma,b)k^{|E(\Gamma)|}}{|{\rm Aut}(\Gamma)|} \frac{N(a,k)!}{\prod_{v\in V_0} N(v) !} \left(\prod_{v\in V_0} N(v)!\right) \zeta_{(\Gamma,b),k\, *}[\PP\oOm_{(\Gamma,b)}(k)]. 
\end{eqnarray}
Then formula~\eqref{for:intmain} is obtained by combining~\eqref{for:pushxi},~\eqref{for:pushpsi},~\eqref{for:pushR2},~\eqref{for:pushR2}.
\end{proof}

\section{Growth of powers of the area form}\label{sec:growth}

In this section, we complete the proof of Theorems~\ref{th:main1} and~\ref{th:main2}. First, we recall that the cohomology class $s_{g,n,d}(a,k)$ defined in the introduction via the area metric on $\mathcal{O}(1)$ is the Segre class of $\oOm_{g,n}(a,k)$. In particular $s_{g,n,0}(a,k)=1$, and $s_{g,n,0}=1$, thus formula~\eqref{formain:base} is valid.

\sskip

If $x$ is a real number then we denote by $\lfloor x \rfloor_k=\frac{1}{k} \lfloor kx\rfloor $. If $a$ is a real vector, and $k$ is a positive integer, then we denote by $\lfloor a \rfloor_k=(\lfloor a_1 \rfloor_k, \ldots, \lfloor a_n \rfloor_k)$. We will use the following lemma to prove formulas~\eqref{formain:small} and~\eqref{formain:der}.
\begin{lemma}\label{lem:growth} Let $f\colon \Delta^{\leq 2}_{g,n}\to H^*(\oM_{g,n},\QQ)$ be a continuous piece-wise polynomial. Let $g\colon\QQ\PP\Delta^{\leq 2}_{g,n}\to H^*(\oM_{g,n},\QQ)$  be a function on the set of rational pairs $\Delta^{\leq 2}_{g,n}$ such that $k(g(a,k)- f(a))$ is bounded. Then, the sequence of functions $f_k\colon \Delta^{\leq 2}_{g,n}\to H^*(\oM_{g,n},\QQ)$ defined by $f_k(a)=g(\lfloor a \rfloor_k,k)$ absolutely converges to $f$ as $k$ goes to infinity.
\end{lemma}

\begin{proof} The function $f$ is $K$-Lipschitz for a positive constant $K$. Besides, there exists $K'>$ such that  $\lVert f(a)-g(a,k)\rVert< K'/k$, so 
$$
\lVert f(a)-g(a,k)\rVert =\lVert (f(a)-f(\lfloor a \rfloor_k))+(\lfloor a \rfloor_k-g(a,k))\rVert < (K+K')/k.$$
\end{proof}

\subsection{Small angles} Let $(a,k)$ be a rational pair of $\Delta_{g,n}^{\leq 1/2}$. The space $\oOm_{g,n}(a,k)$ is isomorphic to the vector bundle over $\oM_{g,n}$ with fibers $H^0(C,\omega_{\rm log}^{\otimes k}(ka_1x_1+\ldots+ka_nx_n))$. The first cohomology group $H^1(C,\omega_{\rm log}^{\otimes k}(ka_1x_1+\ldots+ka_nx_n))$ is trivial, so the Segre class of  $\oOm_{g,n}(a,k)$ can be computed by applying the Grothendieck-Riemann-Roch formula. We recall from ~\cite{Bin} that the Chern characters of $\oOm_{g,n}(a,k)$ are given by 
\begin{equation}
    {\rm ch}_d\left(\oOm_{g,n}(a,k)\right)= \frac{B_{d+1}(k)}{(d+1)!} \kappa_1 - \sum_{i=1}^n \frac{B_{d+1}(ka_i)}{(d+1)!} \psi_i + \frac{1}{(d+1)!} \delta_d
\end{equation}
where $\delta_d$ is a boundary term that does not depend on $k$ or $a$, and $B_d$ is the Bernoulli polynomial. We recall that $B_2=x^2-x+1/6$ and ${\rm deg}\, B_d=d$. For a fixed $a$, we see that ${\rm ch}_d$ is a polynomial of degree $d+1$ in $k$. In particular, as $k$ goes to infinity the Chern characters ${\rm ch}_d$ for $d\geq 2$ do not contribute to the highest degree term in $k$ of the Segre class, so
\begin{eqnarray}
    \nonumber s_{g,n,d}(a,k) = s_d\left(\oOm_{g,n}(a,k)\right) &=& \frac{(-1)^d}{d!}\left(c_1\left(\oOm_{g,n}(a,k)\right)\right)^d + O_{\Delta_{g,n}^{\leq 1/2}}(k^{2d-1}) \\ 
    &=& \frac{k^{2d}}{d!}\left(\frac{1}{2}\kappa_1 - \sum_{i=1}^n \frac{a_i^2}{2} \psi_i\right)^d + O_{\Delta_{g,n}^{\leq 1/2}}(k^{2d-1}),
\end{eqnarray}
where the notation $O_E(k^\ell)$ stands for a function $g\colon \QQ\PP(E)\to H^*(\oM_{g,n},\QQ)$ with norm bounded by $Ck^{\ell}$ for some constant $C$ (and $V$ is an implicitly defined vector bundle, here $H^*(\oM_{g,n},\QQ)$).  Together with Lemma~\ref{lem:growth}, this estimate implies~\eqref{formain:small} of Theorem~\ref{th:main2}.

\subsection{Expression of the derivatives} We fix a triple $(g,n,d)$. We will show that 
$$k^{2g+1}s_{g',n',d'}(a,k)-ks_{g,n,d)}(a)$$ is bounded on $\QQ\PP(\Delta_{g,n}^{\leq 2)}$, where $s_{g,n,d}$ is defined by the relations of Theorem~\ref{th:main2}. We work by induction, so we assume that this holds for all triples $(g',n',d')$  such that $2g'-2+n'<2g-2+n$ or $(g',n')=(g,n)$ and $d'<d$. The bases cases $(g,n,d)=(0,3,0)=(1,1,0)$ have already been treated. 

\sskip

Let $(a,k)$ be a rational pair of $\Delta_{g,n}^{<2}$. Let $d\geq 0$. We multiply formula~\eqref{for:intmain}  by $\xi^{r+d-1}$ and push-forward the result along  $p:\PP\oOm_{g,n}(a',k)\to\oM_{g,n}$. This way, we obtain the following relation
\begin{eqnarray}
    \nonumber s_{g,n,d}(a,k)+ka_n\psi_n s_{g,n,d-1}(a,k) &=& s_{g,n,d}(a',k) \\ && +  \sum_{\Gamma\in {\rm Rat}_{g,n}}\sum_{b\in \Delta_{\Gamma}(a,k)} \frac{m(\Gamma,b)k^{|E(\Gamma)|}}{|{\rm Aut}(\Gamma)|} \zeta_{\Gamma *} s_{d-|E(\Gamma)|}\left(\oOm_{(\Gamma,b)}(k)\right).
\end{eqnarray}
In the sum, the space $\oOm_{(\Gamma,b)}(k)$ is considered as a cone over the stratum of the moduli space of curves $\oM_{\Gamma}$ and the Segre class is given by
\begin{equation}
    s_{d}\left(\oOm_{(\Gamma,b)}(k)\right)=\underset{|\underline{d}|=d}{\sum_{\underline{d}=(d_v)_{v\in V^{\rm out}}}} \left(\bigotimes_{v\in V^{\rm out}} s_{d_v}\left(\oOm_{g(v),n(v)}(b(v),k)\right)\right).
\end{equation}
Therefore, if $(a,k)$ is a rational pair of $\Delta_{g,n}^{\leq 2}$ then
\begin{eqnarray}
    \nonumber s_{g,n,d}(a,k)&=& s_{g,n,d}(a[0],k)+  \sum_{0\leq \ell\leq ka_n-1}s_{g,n,d}(a[(\ell+1)/k],k) - s_{g,n,d}(a[\ell/k],k)\\
    \label{for:intd}&=&  s_{g,n,d}(0,k) +  \sum_{0\leq \ell\leq ka_n-1} \ell \psi_{n}s_{g,n,d-1} \\
    \nonumber&& + \underset{0\leq \ell\leq ka_n-1}{\sum_{\Gamma\in {\rm Rat}_{g,n}}}\sum_{b\in \Delta_{\Gamma}(a[\ell/k],k)} \zeta_{\Gamma *} s_{d-|E(\Gamma)|}\left(\oOm_{(\Gamma,b)}(k)\right).
\end{eqnarray}
We have already shown that 
$$s_{g,n,d}(a[0],k)=k^{2d}e_{g,n,d}(a[0])+O_{\Delta_{g,n}^{\leq 1/2}}(k^{2d-1}).$$
Besides, by induction hypothesis 
\begin{eqnarray*}
     \frac{1}{k^{2d}}\sum_{0\leq \ell\leq ka_n-1} \ell \psi_{n}s_{g,n,d-1}(a,k) &=&  \frac{1}{k^2}\sum_{0\leq \ell\leq ka_n-1} \ell \psi_{n}s_{g,n,d-1}(a) + O_{\Delta_{g,n}^{\leq 2}}(k^{-1})\\
     &=& \int_{t=0}^{a_n} t\, \psi_ns_{g,n,d-1}(a[t])\, dt. 
\end{eqnarray*}
To control the sum in the RHS of~\eqref{for:intd}, we start by fixing a rational graph $\Gamma$. Then, by induction hypothesis, we have
\begin{eqnarray*}
    &&\frac{1}{k^2d-1}\sum_{b\in \Delta_{\Gamma}(a[\ell],k)} m(\Gamma,b)k^{|E(\Gamma)|}\zeta_{\Gamma *} s_{d-|E(\Gamma)|}\left(\oOm_{(\Gamma,b)}(k)\right) \\
    &=& \frac{1}{k^{|E(\Gamma)|-1}} \sum_{\underline{d}\vdash d-|E(\Gamma)|}\sum_{b\in \Delta_{\Gamma} (a,k)} m(\Gamma,b) \zeta_{\Gamma *}\left(\bigotimes_{v\in V^{\rm out}} s_{g(v),n(v),d_v}(b(v))\right) + O_{\Delta_{g,n}^{\leq 1/2}}(k^{-1})\\
    &=& \int_{b\in\Delta_{\Gamma}} \zeta_{\Gamma *}\left(\bigotimes_{v\in V^{\rm out}} s_{g(v),n(v),d_v}(b(v))\right)\, db  + O_{\Delta_{g,n}^{\leq 1/2}}(k^{-1}).
\end{eqnarray*}
Altogether, we obtain the following expression 
\begin{eqnarray*}
    k^{-2d}s_{g,n,d}(a,k) &=& e_{g,n,d}(a[0])+ \int_{t=0}^{a_n} t\, \psi_ns_{g,n,d-1}(a[t])\, dt  \\
    && -  \underset{\underline{d}\vdash d-|E(\Gamma)|}{\sum_{\Gamma \in {\rm Rat}_{g,n}}} \int_{t=0}^a \int_{b\in \Delta_{\Gamma}(a[t])} \frac{m(\Gamma,b)}{|{\rm Aut}(\Gamma)|} \zeta_{\Gamma *}\left(\bigotimes_{v\in V^{\rm out}} s_{g(v),n(v),d_v}(b(v))  \right) \, db \\&& +  O_{\Delta_{g,n}^{\leq 1/2}}(k^{-1}).
\end{eqnarray*}
Therefore $k^{-2d}s_{g,n,d}(a,k)$ converges uniformly and the limit $s_{g,n}$ satisfies~\eqref{formain:der}. This completes the proof of theorems~\ref{th:main1} and~\ref{th:main2}.

\section{Identities at integral singularities}\label{sec:WK}

In this section, we study the volume functions in the presence of integral coordinates. We recall from~\cite{SauFlat} that $\int_{\oM_{g,n}} (-1)^{g-1+n}s_{g,n,3g-3+n}(a,k)=0$ is of the same sign as $\prod_{i=1}^n {\rm sin}(a_i\pi)$. In particular, if $a_n$ is integral, then $V_{g,n}(a)=0$. As shown in Section~\ref{sec:ind}, the function $V_{g,n}$ is $C^1$. This implies that the function ${\rm Vol}_{g,n}$ is continuous and non-negative, thus completing the proof of Theorem~\ref{th:volgn}.

\sskip

The rest of the section will be dedicated to the proof of Theorem~\ref{th:kdv}. Let $(g,n)\neq (0,3)$ be a pair such that $n\geq 1$, and let $a$ be a vector in $\Delta_{g,n}^{\leq 1/2}$. Theorem~\ref{th:kdv} will be deduced from the identity $V_{g,n+1}(a_1,\ldots,a_n,2)=0$. 

\subsection{Expression of $V_{g,n+1}$ in terms of Mirzakhani polynomials} Our first task is to write $V_{g,n+1}$ in terms of Mizakhani polynomials using the closed expression of $s_{g,n}$ proved in Section~\ref{sec:ind}. For all $\ell\geq 0$, and $0\leq t\leq 2$ we set 
\begin{eqnarray*}
    V_{g,n+1}^\ell(a,t)&\coloneqq&\int_{\oM_{g,n+1}} \psi_n^\ell s_{g,n+1}(a_1,\ldots,a_n,t), \text{ and }\\
    P_{g,n+1}^\ell(a,t)&\coloneqq&\int_{\oM_{g,n+1}} \psi_n^\ell c_{g,n+1}(a_1,\ldots,a_n,t).
\end{eqnarray*}
With this convention we have $V_{g,n+1}^0=(-1)^{g+n}V_{g,n+1}$ while $P_{g,n+1}^0=P_{g,n+1}$. We have
\begin{eqnarray}
    V_{g,n+1}^{\ell}(a,t)&=& P_{g,n+1}^\ell(a,t) \text{ if $\ell\geq 2$},\\
    V_{g,n+1}^1(a,t)&=& P_{g,n+1}^1(a,t)-\int_{u=0}^t\mathcal{D}^4(a,u)\,  du\\
    V_{g,n+1}^0(a,t)&=&  \int_{u=0}^t \left(uV_{g,n+1}^1(a,u)- \mathcal{D}^3(a,u)\right)\, du,
\end{eqnarray}
where $\mathcal{D}^4(a,u)=\int_{\oM_{g,n+1}} D^{4,1}_{g,n+1}$$(a,u) du $ and $\mathcal{D}^3(a,u)\int_{\oM_{g,n+1}}D^{3}_{g,n+1}(a,u)$ (the classes $D^{3}_{g,n+1}$ and $D^{4,\ell}_{g,n+1}$ are defined in Section~\ref{sec:ind}). The integral of $D_{g,n+1}^{4,0}$ vanishes, thus does not appear in the expression of $V_{g,n+1}$. Indeed, the contribution of a rational graph with a central vertex with 4 half-edges to $D^{4,0}_{g,n+1}$ is defined as the push-forward of classes supported only on the outer vertices, so the top cohomological degree of this class is trivial.
\sskip

We use the expression of the $D$-functions given at Section~\ref{sec:ind} to re-write the RHS of these formulas with Mirzakhani polynomials. As $\int_{\oM_{0,4}}\psi_n=1$, we have
\begin{eqnarray*}
     \nonumber \mathcal{D}^4(a,u)\!\! &=& \!\! \sum_{1\leq i<j<n+1} (a_i+a_j+u-2)^+ P_{g,n-1}(\ldots,\widehat{a_i},\ldots,\widehat{a_j},\ldots,a_i+a_j+u-2) \\
    \label{for:D4}&& \!\! + \sum_{1\leq i<n+1}  \int_{y=0}^{(a_i+u-2)^+} \frac{y (a_i+u-2-y)}{2} P_{g-1,n+1}(\ldots,\widehat{a_i},\ldots,y,u+a_i-2-y)\, dy\\
    \nonumber&& \!\! + \sum_{1\leq i<n+1} \underset{I_1\sqcup I_2=\{1,\ldots,n\}\setminus \{i\}} {\sum_{g_1+g_2=g}}   \int_{y=0}^{(a_i+u-2)^+} \frac{y (a_i+u-2-y)}{2}  \\
    \nonumber&&\,\,\,\,\,\,\,\,\,\, \,\,\,\,\,\,\,\,\,\,\,\,\,\, \,\,\,\,\,\,\,\,\,\,\,\,\,\, \,\,\,\,
    P_{g_1,|I_1|+1}(y,\{a_i\}_{i\in I_1}) \times P_{g_2,|I_2|+1}(a_i+u-2-y,\{a_i\}_{i\in I_2})\, dy.
\end{eqnarray*}
This expression determines $V_{g,n+1}^1$ in terms of the $P$-functions. If $u\leq 1$, then these sums are trivial, so $V^1_{g,n+1}(a,t)=P^1_{g,n+1}(a,t)$ if $t\leq 1$.
\begin{eqnarray}\nonumber
     \mathcal{D}^3(a,u)\!\! &=& \!\! \sum_{1\leq i<n+1} (a_i+u-1)^+ V^0_{g,n}(\ldots,\widehat{a_i},\ldots,a_i+u-1) \\
    \label{for:D3}&& \!\! +  \int_{y=0}^{(u-1)^+} \frac{y (u-1-y)}{2} V^0_{g-1,n+2}(\ldots,\widehat{a_i},\ldots,y,u-1-y)\, dy\\
    \nonumber&& \!\! + \underset{I_1\sqcup I_2=\{1,\ldots,n\}} {\sum_{g_1+g_2=g}}     \int_{y=0}^{(u-1)^+} \frac{y (u-1-y)}{2}  \\
    \nonumber&&\,\,\,\,\,\,\,\,\,\, \,\,\,\,\,\,\,\,\,\,\,\,\,\, \,\,\,\,\,\,\,\,\,\,\,\,\,\, \,\,\,\,
    V^0_{g_1,|I_1|+1}(y,\{a_i\}_{i\in I_1}) \times V^0_{g_2,|I_2|+1}(u-1-y,\{a_i\}_{i\in I_2})\, dy.
\end{eqnarray}
Here, the function $V$-functions are evaluated at vectors with coordinates at most $3/2$ for the first term at most and $1$ for the others. Besides, if $u\leq 1$, then only the first sum is non-trivial. In particular, we have the following expression for $t\leq 1$:
\begin{equation}\label{for:VP1}
    V^0_{g,n+1}(a,t)=P_{g,n+1}(a,t)-\sum_{i=1}^n \int_{u=0}^{(t+a_i-1)^+} u P_{g,n}(\ldots, \widehat{a_i},\ldots,u)\, du.
\end{equation}
We use this expression, and $V^0_{g,n+1}(a,1)=0$ to re-prove the following result by Do--Norbury~\cite{DoNor}.
\begin{theorem}
    For all $a\in \RR^n$ we have
    \begin{equation}
    P_{g,n+1}(a,t)=\sum_{i=1}^n \int_{u=0}^{a_i} u P_{g,n}(\ldots, \widehat{a_i},\ldots,u)\, du.
\end{equation}
\end{theorem}

Here, we will use~\eqref{for:VP1} to express $\mathcal{D}^3$ in terms of Mirzakhani polynomials. To do so, we denote by $\widetilde{\mathcal{D}}^3(a,u)$ the expression obtained by replacing $V$ by $P$ in the RHS of~\eqref{for:D3}. Then, we set
\begin{eqnarray*}
     \mathcal{D}'(a,t)&\coloneqq & \sum_{1\leq i<n+1}\Bigg( \underset{j\neq i}{\sum_{1\leq j<n+1}} \int_{u=0}^{(a_i+a_j+t-2)^+}\!\!\! u(a_i+t-1) P_{g,n-1}(\ldots,\widehat{a_i},\ldots,\widehat{a_j},\ldots,u)\, du \\
     && \,\,\,\,\,\,\,\,\,\,\,\, +  \int_{u=0}^{(a_i+t-2)^+} \int_{y=0}^{(u-1)^+}\!\!\! \frac{y(u-y)}{2}(a_i+t-1)\\
    &&\,\,\,\,\,\,\,\,\,\, \,\,\,\,\,\,\,\,\,\,\,\,\,\, \,\,\,\,\,\,\,\,\,\,\,\,\,\, \,\,\,\,
    P_{g-1,n+1}(\ldots,\widehat{a_i},\ldots,\ldots,y,u-y)\,dy\, du \\
    && \,\,\,\,\,\,\,\,\,\,\,\, + \underset{I_1\sqcup I_2=\{1,\ldots,n\}} {\sum_{g_1+g_2=g}}  \int_{u=0}^{(a_i+t-2)^+} \int_{y=0}^{(u-1)^+}\!\!\! \frac{y(u-y)}{2}(a_i+t-1)\\
    &&\,\,\,\,\,\,\,\,\,\, \,\,\,\,\,\,\,\,\,\,\,\,\,\, \,\,\,\,\,\,\,\,\,\,\,\,\,\, \,\,\,\,
    P_{g_1,|I_1|+1}(y,\{a_i\}_{i\in I_1}) \times P_{g_2,|I_2|+1}(u-y,\{a_i\}_{i\in I_2})\,dy\, du \Bigg)\\
    && \!\! + \sum_{1\leq i<n+1}  \int_{y=0}^{(t-1)^+} \int_{u=0}^{(y+a_i-1)^+} y u (t-1-y) P_{g-1,n+1}(\ldots,\widehat{a_i},\ldots,u,t-1-y)\, du \, dy\\
    && \!\! + \underset{I_1\sqcup I_2=\{1,\ldots,n\}} {\sum_{g_1+g_2=g}}   \sum_{i \in I_1}  \int_{y=0}^{(t-1)^+} \int_{u=0}^{(y+a_i-1)^+} y u (t-1-y)  \\
    &&\,\,\,\,\,\,\,\,\,\, \,\,\,\,\,\,\,\,\,\,\,\,\,\, \,\,\,\,\,\,\,\,\,\,\,\,\,\, \,\,\,\,
    P_{g_1,|I_1|+1}(u,\{a_i\}_{i\in I_1}) \times P_{g_2,|I_2|+1}(t-1-y,\{a_i\}_{i\in I_2})\, du \, dy,
\end{eqnarray*}
and  
\begin{equation}
    \mathcal{D}''(a,t)\coloneqq  \sum_{1\leq i<j<n+1} \int_{y=1-a_i-a_j}^{{\max}(t-1,1-a_i-a_j)} y(t-1-y)  P_{g,n-1}(\ldots,\widehat{a_i},\ldots,\widehat{a_j},\ldots,t-1-y)\, dy.
\end{equation}
With this notation, we have
\begin{equation}
    \mathcal{D}^{3}(a,t)=\widetilde{\mathcal{D}}^3(a,t)-\mathcal{D}'(a,t)-\mathcal{D}''(a,t).
\end{equation}
To explain this last expression, note that the term $\mathcal{D}'$ corresponds to the correction between the $V$ and $P$ function given by~\eqref{for:VP1}. The term $\mathcal{D}''$ has another explanation: in the third sum of~\eqref{for:D3}, if $g_1=0$ and $I_1=\{i,j\}$, then the integration domain includes values of $y$ such that $y+a_i+a_j\geq 1$ while $\Delta_{0,3}$ is the set of vectors of size at most $1$. To make sense of this expression, we must impose that $V_{0,3}(a)=0$ if $|a|>1$ while $P_{0,3}(a)=1$. The term $\mathcal{D}''$ corresponds to the correction obtained by integration on the complement of $\Delta_{0,3}$. Altogether, we obtain the following relation
\begin{equation}\label{for:VP2}
    V^0_{g,n+1}(a,2)=0=P_{g,n+1}(a,2) - \int_{t=0}^2 \mathcal{D}^{3}(a,t)-  \int_{t=0}^2\int_{u=0}^t t\mathcal{D}^4(a,u)\,du\, dt.
\end{equation}
The RHS of this identity is a linear combination of polynomials in $a$ constructed from Mirzakhani polynomials.

\subsection{Odd part of the relation}  In order to prove Theorem~\ref{th:kdv}, we simplify the identity~\eqref{for:VP2} by extracting the monomials with odd powers in $a_1$. If $P$ is a polynomial in $a_1,\ldots,a_n$, then we denote 
\begin{equation*}
    P^{\rm odd}=\frac{1}{4}\bigg(\! P(a_1,a_2,\ldots,a_n)+P(a_1,-a_2,\ldots,-a_n)- P(-a_1,a_2,\ldots,a_n)-P(-a_1,-a_2,\ldots,-a_n)\!\bigg)
\end{equation*}
the odd part in $a_1$ of the odd degree part of $P$. Mirzkhani polynomials are even in all variables, so~\eqref{for:VP2} implies
\begin{equation}\label{for:VP2odd}
    \left(\int_{t=0}^2 \widetilde{\mathcal{D}}^{3}(a,t) dt\right)^{\rm odd}= \left(\int_{t=0}^2 \mathcal{D}'(a,t)+\mathcal{D}''(a,t) -t \int_{u=0}^t \mathcal{D}^4(a,u)\, du    \, dt\right)^{\rm odd}
\end{equation} 
The parity of Mirzakhani polynomials  also implies the following simple expression of the LHS:
\begin{eqnarray}
    \nonumber \left(\int_{t=0}^2 \widetilde{\mathcal{D}}^{3}(a,t)\, dt\right)^{\rm odd} &=& \left(\int_{t=0}^{2} (a_i+t-1)^+ P_{g,n}(a_1+t-1,a_2\ldots,a_n)\, dt\right)^{\rm odd} \\
    \nonumber&=& \left(\int_{t=0}^{a_i+1} t P_{g,n}(t,a_2\ldots,a_n)\, dt\right)^{\rm odd}\\
    \label{for:dtildeodd} &=& \int_{t=1-a_1}^{1+a_1} t P_{g,n}(t,a_2\ldots,a_n) \, dt.
\end{eqnarray}
Indeed, the other terms are even in the variable $a_1$. 
We introduce notation to re-group the terms in the RHS of~\eqref{for:VP2odd}.
For all $2\leq i\leq n$, we set
\begin{eqnarray*}
    \nonumber\mathcal{D}_{\{1,i,n+1\}}(a)&\coloneqq &  \int_{t=0}^2 \int_{u=0}^{(a_1+a_i+t-2)^+}\!\!\! u\left((a_1+t-1)+(a_i+t-1)\right) P_{g,n-1}(u,\ldots,\widehat{a_i},\ldots)\, du \, dt\\
    && + \int_{t=0}^2 \int_{y=1-a_1-a_j}^{{\max}(t-1,1-a_1-a_i)} y(t-1-y)  P_{g,n-1}(t-1-y,\ldots,\widehat{a_i},\ldots)\, dy \, dt\\
    \nonumber && - \int_{t=0}^2  \int_{u=0}^{t}\!\!\! t(a_1+a_i+u-2)^+ P_{g,n-1}(a_1+a_i+u-2,\ldots,\widehat{a_i},\ldots)\, du \, dt\\
     &=& \int_{t=0}^{a_1+a_i} \int_{u=0}^{t} \left(u(2t-a_1-a_i+2) + (t-u+1-a_1-a_i)u- u(t+2-a_1-a_i) \right)\\
    && \,\,\,\,\,\,\,\,\,\,\,\,\,\,\,\,\,\,\,\,\,\,\,\,\,\,\,\,\,\,\,\,\,\,\,\,\,\,\,\,\,\,\,\,\,\,\,\,\,\,\,\,\,\,\,\,\,\,\,\,\,\,\,\,\,\,\,\,\,\,\,\,\,\,\,\,\,\,\,\,\,\,\,\,\,\,\,\,\,\,\,\,\,\,\,\,\,\,\,\,\,\,\,\,\,\,\,\,\,\,\,\,\,\,\,\,\,\,\,\, P_{g,n-1}(u,\ldots,\widehat{a_i},\ldots) \, du \, dt
\end{eqnarray*}
From the first expression to the second, we have made the change of variable $t\mapsto t+2-a_1-a_i$, and then $y\mapsto t-u+1-a_1-a_i$ in the second integral, and $u\mapsto u+2-a_1-a_i$ in the third one. With this notation, we have
\begin{equation}\label{for:Diodd}
    \mathcal{D}_{\{1,i,n+1\}}(a)^{\rm odd}= \int_{t=a_i-a_1}^{a_1+a_i} \int_{u=0}^{t} u P_{g,n-1}(u,\ldots,\widehat{a_i},\ldots) \, du \, dt
\end{equation}
\begin{eqnarray*}
    \mathcal{D}_{\{1,n+1\}}^{\rm loop}(a)\!\!\! &\coloneqq & \!\! \int_{t=0}^2 \int_{y=0}^{(t-1)^+} \int_{u=0}^{(y+a_1-1)^+} y u (t-1-y) P_{g-1,n+1}(\ldots,u,t-1-y)\, du \, dy\, dt\\ 
    && \!\! + \int_{t=0}^2 \int_{u=0}^{(a_1+t-2)^+} \int_{y=0}^{u} \frac{y(u-y)}{2}(a_1+t-1) P_{g-1,n+1}(\ldots,y,u-y)\, dy\,du\,dt \\
    && \!\! -\int_{t=0}^2 \int_{u=0}^t \int_{y=0}^{(a_1+u-2)^+} \frac{ty (a_1+u-2-y)}{2} P_{g-1,n+1}(\ldots,y,u+a_1-2-y)\, dy\,du\,dt \\
     &=& \!\! \int_{t=0}^{a_1} \int_{y=0}^{t} \int_{u=0}^{y} (y+1-a_1) u (t-y) P_{g-1,n+1}(\ldots,u,t-y)\, du \, dy\, dt\\ 
    && \!\! + \int_{t=0}^{a_1} \int_{u=0}^{t} \int_{y=0}^{u} \frac{y(u-y)}{2}(t+1) P_{g-1,n+1}(\ldots,y,u-y)\, dy\,du\,dt \\
    && \!\! -\int_{t=0}^{a_1} \int_{u=0}^t \int_{y=0}^{u} \frac{(t+2-a_1)y (u-y)}{2} P_{g-1,n+1}(\ldots,y,u-y)\, dy\,du\,dt \\ 
    &=& \!\! \int_{t=0}^{a_1} \int_{u=0}^{t} \int_{y=0}^{u} \left(t-y+1-a_1+\frac{a_1-1}{2}\right) (u-y) y \\
    &&\,\,\,\,\,\,\,\,\,\,\,\,\,\,\,\,\,\,\,\,\,\,\,\,
    \,\,\,\,\,\,\,\,\,\,\,\,\,\,\,\,\,\,\,\,\,\,\,\,
    \,\,\,\,\,\,\,\,\,\,\,\,\,\,\,\,\,\,\,\,\,\,\,\,P_{g-1,n+1}(\ldots,y,u-y)\, du \, dy\, dt
\end{eqnarray*}
From the first line to the second, we have made the change of variables $t\mapsto t+2-a_1$ in all integrals,  then $y\mapsto t-y+1-a_1$ in the first integral and  $u\mapsto u+2-a_1$ in the third integral. From the second line to the second, we have changed variables in the first integral $u\mapsto u+y$. The odd part of this polynomial is given by
\begin{equation}\label{for:Dloopodd}
    \left(\mathcal{D}_{\{1,n+1\}}^{\rm loop}(a)\right)^{\rm odd} = \!\! \int_{t=-a_1}^{a_1}  \int_{u=0}^{t} \int_{y=0}^{u} \frac{(u-y) y}{2} P_{g-1,n+1}(a_2,\ldots,y,u-y)\, du \, dy\, dt,
\end{equation}
For $g_1+g_2=g$ and $I_1\sqcup I_2=\{2,\ldots,n-1\}$, we define and compute  $\mathcal{D}_{\{1,n+1\}}^{g_1,g_2,I_1,I_2}$ similarly to $\mathcal{D}_{\{1,n+1\}}^{\rm loop}(a)$, so we simply give the final expression of the odd part:
\begin{eqnarray}\label{for:Dsepodd}
    \left(\mathcal{D}_{\{1,n+1\}}^{g_1,g_2,I_1,I_2}(a)\right)^{\rm odd} &=& \!\! \int_{t=-a_1}^{a_1}  \int_{u=0}^{t} \int_{y=0}^{u} \frac{(u-y) y}{2} \nonumber \\
    &&\,\,\,\,\,\,\,\,\,\,\,\,\,\,\,\,\,\,\,\,\,\,\,\,
    \,\,\,\,\,\,\,\, P_{g_1,|I_1|+1}(u,\{a_i\}_{i\in I_1}) \times P_{g_2,|I_2|+1}(u-y,\{a_i\}_{i\in I_2})\, du \, dy\, dt.
\end{eqnarray}
With these functions, we can rewrite~\eqref{for:VP2odd} as
\begin{equation}
     \left(\int_{t=0}^2 \widetilde{\mathcal{D}}^{3}(a,t)\, dt\right)^{\rm odd}=\Bigg(\sum_{i=2}^n \mathcal{D}_{\{1,i,n+1\}}(a) +  \mathcal{D}_{\{1,n+1\}}^{\rm loop}(a) + \underset{ I_1\sqcup I_2 =\{2,\ldots,n\}}{\sum_{g_1+g_2=g}}  \mathcal{D}_{\{1,n+1\}}^{g_1,g_2,I_1,I_2}(a)  \Bigg)^{\rm odd}.
 \end{equation}
Putting~\eqref{for:dtildeodd},~\eqref{for:Diodd},~\eqref{for:Dloopodd}, and ~\eqref{for:Dsepodd} together we obtain the following relation
\begin{eqnarray}\label{eq:kdvint} \nonumber
&&\!\!\!\!\!\!\!\!\!\!\!\int_{t=1-a_1}^{1+a_1} (1+t)P_{g,n}(1+t,a_2,\ldots,a_n)  \\  & &= \sum_{1<i\leq n}  \int_{t=-a_1}^{a_1} \int_{u=0}^{a_i+t} uP_{g,n-1}(u,a_2, \ldots, \widehat{a_i},\ldots)\, dt \\ \nonumber
&&\,\,\,\,\,\,\,\,\, + \int_{t=-a_1}^{a_1} \int_{y=0}^u \frac{y (u-y)}{2} P_{g-1,n+1}( y,u-y,a_2,\ldots) \, dy\, dt\\ \nonumber
&&\,\,\,\,\,\,\,\,\, + \underset{ I_1\sqcup I_2 =\{2,\ldots,n\}}{\sum_{g_1+g_2=g}} \int_{y=0}^u \frac{y (u-y)}{2} P_{g_1,|I_1|+1}(y, \{a_i\}_{i\in I_1}) P_{g_2,|I_2|+1}(u-y, \{a_i\}_{i\in I_2}) \, dy\, dt.
\end{eqnarray}
Theorem~\ref{th:kdv} follows if we take the derivative of this relation with respect to $a_1$.


\newcommand{\etalchar}[1]{$^{#1}$}

\end{document}